\documentclass{amsart}

\newcommand{\Z}{\ensuremath{\mathbf Z}}

\newcommand{\N}{\ensuremath{ \mathbf N }}

\newcommand{\G}{\ensuremath{ \mathbf{G} }}
\newcommand{\X}{\ensuremath{ \mathbf{X} }}

\newtheorem{theorem}{Theorem}
\newcommand{\bt}{\begin{theorem}}
\newcommand{\et}{\end{theorem}}

\newtheorem{lemma}{Lemma}
\newtheorem{corollary}{Corollary}
\newcommand{\bl}{\begin{lemma}}
\newcommand{\el}{\end{lemma}}
\newcommand{\bc}{\begin{corollary}}
\newcommand{\ec}{\end{corollary}}

\newtheorem{problem}{Problem} 
\newcommand{\bprob}{\begin{problem}}
\newcommand{\eprob}{\end{problem}}

\newtheorem{thm}{Theorem}
\newtheorem{conj}{Conjecture}
\newtheorem{pbm}{Problem} 
\newtheorem{defi}{Definition}

\newcommand{\bpf}{\begin{proof}}
\newcommand{\epf}{\end{proof}}

\newcommand{\beq}{\begin{equation}}
\newcommand{\eeq}{\end{equation}}
\newcommand{\benum}{\begin{enumerate}}
\newcommand{\eenum}{\end{enumerate}}

\newcommand{\Prob}{\ensuremath{\text{Prob}}}

\title[Problems in Additive Number Theory, III]
{Problems in Additive Number Theory, III: Thematic Seminars
at the Centre de Recerca Matem\`atica}

\author{Melvyn B. Nathanson}

\thanks{This paper is based on lectures at the Centre de Recerca Matem\`atica in Barcelona on January 23 and January 25, 2008.  I thank Itziar Bardagi and Szabolcs Francsali  for taking notes at these lectures and preparing a preliminary draft of this paper.}

\address{Department of Mathematics, Lehman College (CUNY), 
Bronx, NY 10468, and CUNY Graduate Center, New York, NY 10016}

\email{melvyn.nathanson@lehman.cuny.edu}



\date{\today}

\begin{document}

\begin{abstract}
This is a survey of open problems in different parts of combinatorial and additive number theory.
\end{abstract}

\maketitle

\section{What sets are sumsets?}
Let $\N, \N_0$, $\Z$, and $\Z^d$ denote, respectively, the sets of positive integers,  nonnegative integers, integers, and $d$-dimensional integral lattice points.  Let $\G$ denote an arbitrary abelian group and let $\X$ denote an arbitrary abelian semigroup, written additively.  
Let $|S|$ denote the cardinality of the set $S$.  
For any sets $A$ and $B$,  we write $A\sim B$ if  their symmetric difference is finite, that is, if
$\left|(A\setminus B)\cup(B\setminus A)\right|<\infty$.

\begin{defi}
Let $A$ and $B$ be nonempty subsets of an additive abelian semigroup $\X$. 
The most important definition in additive number theory is the \emph{sumset $A+B$ of the sets $A$ and $B$}:
\[
A+B=\left\{a+b :a\in A, b\in
B\right\}.
\]
\end{defi}

Let $A+B = \emptyset$ if $A= \emptyset$ or $B = \emptyset$.  If $h \geq 3$ and $A_1,\ldots,A_h$ are subsets of \X, then we construct  the sumset $A_1+\cdots+A_h$ inductively as follows:
\begin{align*}
A_1+\cdots+A_{h-1} + A_h 
& = \left( A_1+\cdots+A_{h-1} \right) + A_h \\
& = \left\{a_1+\cdots+a_h : a_i\in A_i \text{ for all } i=1,\dots,h\right\}.
\end{align*}
If $A_1=A_2=\cdots=A_h=A,$ then 
\[
h A = \underbrace{A+ \cdots + A}_{h \text{ times}}
\]
is called the
\emph{$h$-fold sumset of $A$}.

\begin{defi}
Let $A$ be a nonempty subset of an additive abelian semigroup $\X$. 
The set $A$ is called a \emph{basis of order $h$ for \X\ } if $hA = \X.$  
The set $A$ is called an \emph{asymptotic basis of order $h$ for \X\ } if $hA \sim \X.$ 
\end{defi}

A basic problem is: What sets are sumsets?  More precisely,
\begin{pbm}
Given a set $S\subseteq\X$, do there exist sets $A_1,\dots,A_h$ of
integers such that $A_1+\dots+A_h=S$ or $A_1+\dots+A_h\sim S$?
\end{pbm}

\begin{pbm}
Given a set $S\subseteq\X$, does there exist a set $A$ of
integers such that $h A=S$ or $h A\sim S$?
\end{pbm}

These problems are particularly important in the classical cases $\X = \N_0$ ,  $\X =\Z$, and $\X =\Z^d.$

\section{Describing the structure of $h A$ as $h\rightarrow\infty$}

\begin{defi}
Let $A$ be a set of nonnegative integers.  The \emph{counting function} $A(x)$ of the set $A$ counts the number of positive elements of $A$ not exceeding $x,$ that is, $A(x)=\left|A\cap[1,x]\right|$.
The \emph{lower asymptotic density} of $A$ is
\[
d_L(A)=\liminf_{x\rightarrow\infty}\frac{A(x)}{x}.
\]
\end{defi}

\begin{thm}[Nash-Nathanson~\cite{nath85a}]
If $A$ is a set of nonnegative integers with $\gcd(A) = 1$ such that the sumset $h_0 A$ has positive lower asymptotic density for some positive integer $h_0$, then there exists a number $h\in\N$ such that
$h A \sim \N_0$.
\end{thm}

Equivalently, a set $A$ of nonnegative integers is an asymptotic basis of finite order if and only if $\gcd(A) = 1$  and $d_L(h_0 A)>0$ for some positive integer $h_0$.

Let $A$ be a set of nonnegative integers with $0 \in A$  and $\gcd(A)=1.$   We have the increasing sequence of sets 
\[
A\subseteq 2A\subseteq\cdots\subseteq h A\subseteq(h+1)A\subseteq\cdots.
\]
If some sumset has positive lower asymptotic density, then this sequence  becomes eventually constant and equal to $A\setminus \mathcal{F}$ for some finite set $\mathcal{F}$ of integers.  
An important unsolved question is the following:
\begin{pbm}
Suppose that $d_L(h A)=0$ for all $h\ge 1$. Describe the evolution of the structure of the sumset  $h A$ as $h\rightarrow\infty$.
\end{pbm}

\section{Representation functions}
Let $\mathcal{A} = (A_1,\ldots, A_h)$ be an $h$-tuple of subsets of an additive abelian semigroup \X.  We want to count the number $R_{\mathcal{A}}(x)$ of representations of an element $x\in \X$ in the form $x=a_1+\dots+a_h$ with  $a_i\in A_i$ for $i = 1,\ldots, h.$
We discuss here only the special case when $A_i=A$ for $i=a,\ldots,h.$  We shall consider two different representation functions.

\begin{defi}
The \emph{ordered representation function of order $h$} for the set $A$  is the function $R_{A,h}:\X \rightarrow \N_0 \cup \{\infty\} $ defined by 
$$
R_{A,h}(x)=\left|\left\{(a_1,\dots,a_h)\in A^h : x=a_1+\dots+a_h\right\}\right|.
$$
\end{defi}
\begin{defi}
Let $(a_1,\dots,a_h)\in A^h$ and $(a'_1,\dots,a'_h)\in A^h$ be $h$-tuples that represent $x$, that is, $x=a_1+\dots+a_h = a'_1+\dots+a'_h$.  These representations are called \emph{equivalent} if there is a permutation $\sigma$ of the set $\{1,2,\ldots,h\}$ such that $a'_i = a_{\sigma(i)}$ for $i=1,2,\ldots,h.$  
The \emph{unordered representation function $r_{A,h}(x)$ of order $h$}  counts the number of equivalence classes of representations of $x$.  \end{defi}

If  \X\ is a linearly ordered semigroup such as $\N_0$ or \Z, then we can write 
$$
r_{A,h}(x)=\left|\left\{(a_1,\dots,a_h)\in A^h :
x=a_1+\dots+a_h \text{ and } a_1\le a_2\le\dots\le a_h\right\}\right|
$$

Let $\mathcal{F}(\X)$ denote the set of all functions
$f:\X\rightarrow \N_0 \cup \{\infty\}$, and let
\[
\mathcal{R}_h(\X)=\left\{r_{A,h} : A\subseteq\X \right\}
\]
denote the set of all unordered representation functions order $h$ of subsets of \X.  Then $\mathcal{R}_h(\X) \subseteq\mathcal{F}(\X)$.
A simple question is:
\begin{pbm}
What functions are representation functions?
\end{pbm}

This problem seems hopelessly difficult at this time.  We consider the special case of representation functions of asymptotic bases.  
Define the function space
\[
\mathcal{F}^{(0)}(\X)=\left\{f:\X \rightarrow  \N_0 \cup \{\infty\}:
\left|f^{-1}(0)\right|<\infty\right\}.
\]
This is the space of functions on \X\ with only finitely many zeros.  Let
\[
\mathcal{R}^{(0)}_h(\X)=\left\{r_{A,h} : A\subseteq\X\ \text{ and } hA \sim \X \right\}
\]
be the set of representation functions of asymptotic bases of order $h$ for \X.  Then 
\[
\mathcal{R}^{(0)}_h(\X) \subseteq \mathcal{F}^{(0)}(\X).
\]

\begin{pbm}
What functions in $\mathcal{F}_0(\X)$ are representation functions of asymptotic bases?  Equivalently, what are necessary and sufficient conditions for a function on \X\ with only finitely many zeros to be the representation function of an additive basis?  
\end{pbm}

For the group \Z\ of integers there is the following amazing result.

\begin{thm}[Nathanson~\cite{nath03a,nath04a,nath05a}]
For every integer $h \geq 2$ and for every function $\mathcal{F}_0(\Z)$ there exists a set $A\subseteq\Z$  such that $f=r_{A,h}$.  Equivalently,
\[
\mathcal{F}_0(\Z) = \mathcal{R}^{(0)}_2(\Z) = \mathcal{R}^{(0)}_3(\Z)= \cdots = \mathcal{R}^{(0)}_h(\Z) =  \cdots.
\]
\end{thm}

The classification of arbitrary representation functions for the integers is still open.

\begin{pbm}
What functions in $\mathcal{F}(\Z)$ are representation functions of sumsets of order $h$?
\end{pbm}

Very little is known about representation functions of sums  of sets of nonnegative integers, even for sets that are asymptotic bases.   G. A. Dirac gave an elegant proof of the following beautiful result.

\begin{thm}[Dirac~\cite{dira51}]
If $A$ is an asymptotic basis of order 2 for $\N_0$, then the unordered representation function $r_{A,2}$
is not eventually constant.
\end{thm}

\begin{proof}
The generating function  $G_A(z)=\sum_{a\in A}z^a$ converges in the open unit disc and diverges as $z\rightarrow 1^{-}.$   We have $G_A^2(z)=\sum_{n=0}^{\infty}R_A(n)z^n$
and
\[
\frac{G_A^2(z)+G_A(z^2)}{2}=\sum_{n=0}^{\infty}r_{A,2}(n)z^n.
\]
If $r_{A,2}(n) =  c$ for all $n\ge n_0$, then there is a polynomial $P(z)$ such that 
$$
\frac{G_A^2(z)+G_A(z^2)}{2} =P(z)+\frac{c}{1-z} .
$$
Let $0<x<1$ and  $z=-x$.  Then 
\[
G_A(x^2) \leq G_A^2(-x)+G_A(x^2)=2P(-x)+\frac{2c}{1+x}.
\]
As $x\rightarrow 1^{-}$, the right side approaches $2P(-1)+c$ while the left side diverges to $\infty,$ which is absurd.
\end{proof}

One of the most famous and tantalizing unsolved problems in additive number theory is the following: 

\begin{pbm}[Erd\H{o}s-Tur\'an~\cite{erdo-tura41}]
Let $h \geq 2.$  Prove that if $A$ is an asymptotic basis of order $h$ for the nonnegative integers, then the representation function $r_{A,h}$ must be unbounded.
\end{pbm}

\section{Sets with more sums than differences}
Let $A$ be a finite set of integers. We define the sumset 
\[
A+A=\left\{a+a' : a,a'\in A\right\}
\]
and the difference set 
\[
A-A=\left\{a-a' : a,a'\in A\right\}.
\]
Since $2+3=3+2$ but $2-3\ne 3-2$,
that is, since addition is commutative but subtraction is not commutative, 
it would be reasonable to conjecture that $\left|A+A\right|\le\left|A-A\right|$ for every finite set $A$ of integers.  
In some special cases, for example, if $A$ is an arithmetic progression or if $A$ is
symmetric (that is, if $A=c-A$ for some $c\in \Z$), then
$\left|A+A\right|=\left|A-A\right|$.  The ``conjecture,'' however, is false.  The simplest counterexample is the set 
\[
A^*=\{0,2,3,4,7,11,12,14\}
\]
for which $\left|A^* +A^* \right| = 26 > 25 =  \left|A^* -A^*\right|.$ 
This set is ``almost'' symmetric.  We have 
$A^* =A'\cup\{4\}$ where $A'=\{0,2\}\cup\{3,7,11\}\cup\{12,14\}$
 satisfies  $A'=14-A'$.  
 A set with more sums than differences is called an \emph{MSTD set}.  
 Nathanson~\cite{nath07b} showed that if
 $k \geq 3$ and 
$A'=\{0,2\}\cup\{3,7,11,15,19,23,\dots,4k-1\}\cup\{4k,4k+2\}$ and
$A^*=A'\cup\{4\}$ then $A^*$ is an MSTD set.  
Many other examples of MSTD sets have been constructed by Hegarty~\cite{hega07a}.

Finite sets $A$ of integers with the property that $|A+A| > |A - A|$  are extremely interesting, since a sumset really should have more elements than the corresponding difference set.  

\begin{conj}
$\left|A+A\right|\le\left|A-A\right|$ for almost all finite sets $A$.
\end{conj}

Martin and O'Bryant~\cite{mart-obry07a} studied the uniform probability measure on the set of all subsets of $\{1,\ldots, N\},$ that is, they assigned to each subset the probability $2^{-N}.$   Counting sets in this way, they calculated that the average cardinality of a sumset was
\[
\frac{1}{2^N}\sum_{A\subseteq\{1,\dots,N\}}\left|A+A\right|=2N-11
\]
and the average cardinality of a difference set was 
\[
\frac{1}{2^N}\sum_{A\subseteq\{1,\dots,N\}}\left|A-A\right|=2N-7.
\]
Thus, on average, a difference set contains four more elements than the sumset.  However, they also proved the following result.

\begin{thm}[Martin-O'Bryant~\cite{mart-obry07a}]
With the uniform probability measure, there exists a $\delta>0$ such that
$$
\Prob\left(\left|A+A\right|>\left|A-A\right| : A\subseteq\{1,\dots,N\}\right)>\delta
$$
for all $N\ge N_0$
\end{thm}

Thus, choosing a uniform probability measure, it is not true that almost all sets have more differences than sums.  Of course, with the uniform probability distribution most subsets of the interval $\{1,\ldots, N\}$ are large and satisfy $|A+A| = |A-A| = 2N-1$.  This skews the calculation. 

Using a binomial probability distribution, Hegarty and Miller obtained a very different result.

\begin{thm}[Hegarty-Miller~\cite{hega-mill07}]
Let $p:\N \rightarrow (0,1)$ be a function such that $\lim_{N\rightarrow\infty} p(N)=0$ and $\lim_{N\rightarrow\infty} Np(N)= \infty.$  Define the function $q:\N \rightarrow (0,1)$ by $q(N)=1-p(N).$  Consider the binomial probability distribution with parameter $p(N)$ on the space of all subsets of $\{1,\ldots, N\},$  so that a subset of size $k$ has  probability $p(N)^k q(N)^{N-k}.$  Then
$$
\lim_{N\rightarrow\infty} \Prob\Big(\left|A+A\right|>\left|A-A\right| : A\subseteq\{1,\dots,N\}\Big) = 0.
$$
\end{thm}
These theorems seem to contradict each other, but they do not because
they use different  probability measures.

\begin{pbm}
A difficult and subtle problem is to decide what is the appropriate method of counting (or, equivalently, the appropriate probability measure) to apply to MSTD sets.  
\end{pbm}

\section{Comparative theory of linear forms}
Let  $f(x_1,\ldots, x_m)$ be an integer-valued function on the integers and let $A$ be a set of integers.   We define the set 
\[
f(A)=\left\{  f(a_1,\ldots, a_m) : a_1, \ldots, a_m \in A\right\}.
\]
In particular, if $f(x_1,x_2)=u_1 x_1 + u_2 x_2$ is a linear form with nonzero integer coefficients, then 
\[
f(A)=\left\{u_1 a_1 + u_2 a_2 : a_1, a_2 \in A\right\}.
\]
For example, if $s(x_1,x_2)=x_1 + x_2$, then  $s(A)$ is the sumset $A+A.$  If $d(x_1,x_2)=x_1 - x_2$, then $d(A)$ is the difference set $A-A.$  

The binary linear forms $f(x_1,x_2)=u_1 x_1 + u_2 x_2$ and $g(x_1,x_2)=v_1 x_1 + v_2 x_2$ are \emph{related} if 
\begin{align*}
(v_1,v_2) & = (u_2,u_1), \text{ or} \\
(v_1,v_2) & = (du_1, du_2) \text{ for some integer $d$, or} \\
(v_1,v_2) & = (u_1/d, u_2/d) \text{ for some integer $d$ that divides $u_1$ and $u_2$.}
\end{align*}
The binary linear forms $f(x_1,x_2)=u_1 x_1 + u_2 x_2$ and $g(x_1,x_2)=v_1 x_1 + v_2 x_2$ are \emph{equivalent} if there is a finite sequence of binary linear forms $f_0,f_1,\ldots, f_k$ such that $f=f_0$, $g=f_k$, and $f_{i-1}$ is related to $f_i$ for all $i=1,\ldots, k.$  
If $f$ and $g$ are equivalent forms, then $|f(A)| = |g(A)|$ for every finite set $A$.  

The binary linear form $f(x_1,x_2)=u_1 x_1 + u_2 x_2$ is \emph{normalized} if $\gcd(u_1,u_2) = 1$ and $u_1 \geq |u_2| \geq 1.$  Every binary linear form is equivalent to a unique normalized form.

The following is the basic result in the comparative theory of linear forms.

\begin{thm}[Nathanson-O'Bryant-Orosz-Ruzsa-Silva~\cite{nath07d}]
Let $f(x_1,x_2)=u_1 x_1 + u_2 x_2$ and $g(x_1,x_2)=v_1 x_1 + v_2 x_2$ be distinct normalized linear forms.  There exist sets $A$ and $A'$ of integers such that
$\left|f(A)\right|<\left|g(A)\right|$ and
$\left|f(A')\right|>\left|g(A')\right|$.
\end{thm}

Consider now linear forms  in more than two variables.  

\begin{pbm}
Let $f(x_1,\dots,x_m)=u_1x_1+\dots+u_mx_m$  and $g(x_1,\dots,x_m)=v_1x_1+\dots+v_mx_m$ be linear forms with nonzero integer coefficients  in $m \geq 3$ variables.  Suppose that $\gcd(u_1,\ldots, u_m) =  \gcd(v_1,\ldots, v_m)  = 1$ and that $g$ cannot be obtained from $f$ by some permutation of the coefficients or by multiplication by $-1$.  
Does there exist a finite set $A$ of integers such that
$\left|f(A)\right|>\left|g(A)\right|$?
\end{pbm}

\begin{pbm}
Let $f(x_1,\dots,x_m)$  and $g(x_1,\dots,x_m)$ be polynomials with integer coefficients in $m \geq 2$ variables.  
Under what conditions does there exist a finite set $A$ of integers such that
$\left|f(A)\right|>\left|g(A)\right|$?
\end{pbm}

\begin{defi}
Let $f$ be an integer-valued function defined on \Z.  Define
$N_f(k)=\min\left\{\left|f(A)\right| : A\subseteq\Z \text{ and } \left|A\right|=k\right\}$.
\end{defi}

\begin{pbm}
Let $f$ be an integer-valued function defined on \Z, for example, a linear form or a polynomial with integer coefficients.  
Determine $N_f(k)$ and describe the structure of the minimizing sets.
\end{pbm}

\begin{thm}[Bukh~\cite{bukh07}]
Let $f(x_1,\dots,x_m)=u_1x_1+\dots+u_mx_m$  be a linear form with nonzero integer coefficients  in $m \geq 2$ variables.  If $\gcd(u_1,\ldots, u_m) =  1$, then 
$N_f(k)=\Big(\left|u_1\right|+\dots+\left|u_2\right|\Big)k-o(k)$
\end{thm}

\begin{thm}
If $f(x_1,x_2)=x_1+x_2$ then $N_f(k)=2k-1$ and the minimizing sets are finite arithmetic progressions.  Equivalently, if $A$ is a finite set of integers, then
$\left|A+A\right|\ge 2\left|A\right|-1$ and
$\left|A+A\right|=2\left|A\right|-1$ if and only if $A$ is an
arithmetic progression.
\end{thm}

An \emph{affine transform} of a set $A$ of real numbers is a set obtained from $A$ by a sequence of translations and dilations.

\begin{thm}[Cilleruelo-Silva-Vinuesa~\cite{cill-silv-vanu08}]
If $f(x_1,x_2)=x_1+2x_2$,  then $N_f(k)=3k-2$.  Moreover,
$\left|f(A)\right|=3\left|A\right|-2$ if and only if $A$ is an
arithmetic progression.

If $f(x_1,x_2)=x_1+3x_2$, then $N_f(k)=4k-4$.  Moreover,
$\left|f(A)\right|=4\left|A\right|-4$ if and only if $A$ is $\{0,1,3\}$ or $\{0,1,4\}$ or $\left\{0,3,6,\dots, 3\ell - 3\right\}\cup\left\{1,4,7,\dots, 3\ell - 2 \right\}$, or an affine transform of one of these sets.
\end{thm}

\begin{defi}
Let $\mathcal{U}=(u_1,\dots,u_m)$ be a sequence of positive integers.  
A \emph{subsequence sum} of $\mathcal{U}$ is a nonnegative integer of the form $\sum_{i\in I}u_i$, where $I$ is a subset of $\left\{1,\dots,m\right\}$. 
Let $S(\mathcal{U})=\left\{\sum_{i\in
I}u_i :   I \subseteq\left\{1,\dots,m\right\}\right\}$ denote the
set of all subsequence sums of the sequence $\mathcal{U}$.
\end{defi}

 A subsequence sum is 0 if and only if $I = \emptyset.$  
If $\mathrm{U}=u_1+\dots+u_m$, then 
$S(\mathcal{U})\subseteq[0,\mathrm{U}]$.

\begin{defi}
The sequence $\mathcal{U}$ is called \emph{complete} if
$S(\mathcal{U})=[0,\mathrm{U}]$.
\end{defi}

For example, the sequences $(1,1)$ and $(1,2)$ are complete but $(1,3)$ is not complete.

\begin{thm}[Nathanson~\cite{nath08b}]
Let $\mathcal{U}=(u_1,\dots,u_m)$ be a complete sequence of positive integers, and  let
$f(x_1,\dots,x_m)=u_1x_1+\dots+u_mx_m$ be the associated linear form.  If  $U=u_1+\dots+u_m$, then $N_f(k)=Uk - U+1$  for all positive integers $k$.  Moreover, $\left|A\right|=k$ and $\left|f(A)\right|=N_f(k)$ if and only if $A$
is an arithmetic progression of length $k$.
\end{thm}

There is also the dual problem of describing the finite sets of integers whose images under linear maps are large.

\begin{defi}
Let $f$ be an integer-valued function defined on \Z.  Define
$M_f(k)=\max\left\{\left|f(A)\right| : A\subseteq\Z,\;\left|A\right|=k\right\}$.
\end{defi}

\begin{pbm}
Let $f$ be an integer-valued function in $m$ variables defined on \Z.  Determine $M_f(k)$ and describe the structure of the maximizing sets.
For what functions $f$ is $M_f(k) < k^m$?
\end{pbm}

\section{The fundamental theorem of additive number theory}

Let $A=\{a_0<a_1<\ldots<a_{k-1}\}$ be a finite set of integers. 
Consider the shifted set $A'=A-\{a_0\} =\{0<a_1-a_0<\ldots<a_{k-1}-a_0\}$.  
Let $d=\gcd\{a_i-a_0\ :\ i=1,\ldots,k-1\}$, and construct the set 
\[
A^{(N)}=\frac{1}{d}\ast A'=\left\{0<\frac{a_1-a_0}{d}<\ldots<\frac{a_{k-1}-a_0}{d}\right\}.
\]
If $hA$ is the $h$-fold sumset of $A$, then $hA'=hA-\{ha_0\}$ and  
\[
hA^{(N)}=\frac{1}{d}\ast hA'=\frac{1}{d} \ast (hA-\{ha_0\}).
\]
In particular,  $|hA|=|hA^{(N)}|$.

The set $A^{(N)}$ is called the \emph{normalized} form of the set $A$.  In general, a finite set $A$ of integers is normalized if $A=\{0\}$, or if $|A| \geq 2,$  $\min(A)=0,$ and $\gcd(A)=1.$

The following result is often called the Fundamental Theorem of Additive Number Theory.

\begin{thm}[Nathanson~\cite{nath72e}]
Let $A=\{0<a_1<\ldots<a_{k-1}\}$ be a normalized finite set of integers. There exist a positive  integer $h_0$, nonnegative integers $C$ and $D$, and finite sets $\mathcal{C} \subseteq[0,C-2]$ and $\mathcal{D}\subseteq [0,D-2]$ such that
\[
hA=\mathcal{C} \cup [C, ha_{k-1}-D]\cup (\{ha_{k-1} \} -\mathcal{D})
\]
for all $h\geq h_0.$
\end{thm}

If $A$ is a set of nonnegative integers that contains 0, then 
\[
A \subseteq 2A \subseteq \cdots \subseteq hA \subseteq (h+1)A \subseteq \cdots
\]
for all $h \geq 1,$  and the set 
\[
\Sigma(A) = \bigcup_{h=1}^\infty hA
\]
is the additive subsemigroup of the nonnegative integers generated by the set $A$.  
The fundamental theorem implies that if $\gcd(A)=1,$ then 
$$
\Sigma(A)  = \mathcal{C}\cup
[C,\infty).$$  
The number $C-1$ is the largest integer that
cannot be represented as a nonnegative integral linear combination
of $a_1,\ldots,a_{k-1}$. This is called the \emph{Frobenius number} of the set $A$.  Note that $D-1$ is the Frobenius number of the symmetric normalized set
\[
A^{\sharp} = \{a_{k-1}\} - A = \{0  < a_{k-1} - a_{k-2} < \cdots < a_{k-1} - a_{1} < a_{k-1}  \}.
\]
Also, since $\Sigma(A)$ is a semigroup, it follows that if $u$ and $v$ are nonnegative integers with $u+v=C-1,$ then either $u\notin \Sigma(A)$ or $v \notin \Sigma(A).$  Therefore, 
\[
\left| \mathbf{N}_0 \setminus \Sigma(A) \right| = \left|[0,C-1]\setminus \mathcal{C} \right| \geq \frac{C}{2}.
\]

\begin{pbm}
Let $A$ be a normalized finite set of nonnegative integers.  Compute $\mathbf{N}_0 \setminus \Sigma(A),$ that is, the set of numbers that cannot be represented as nonnegative  integral linear combinations of the elements of $A$.  
\end{pbm}

\section{Thin asymptotic bases}

Let $A$ be an infinite set of nonnegative integers that is an asymptotic basis of order $h.$   There is a nonnegative integer $n_0$ such that, if $n_0 \leq n\leq x$, then there exist $a_1,\ldots, a_h \in A$ with $n=a_1+\cdots+a_h$ and  $0\leq a_i\leq n\leq x$ for $i=1,\ldots, h.$  
Denote by $A(x)$ the counting function of the set $A$.  
Since the interval $[n_0,x]$ contains at least $x-n_0$ nonnegative integers, it follows that $(A(x)+1)^h\geq x-n_0$ and so 
\[
A(x)\gg x^{1/h}
\]
for every asymptotic basis $A$ of order $h$.

\begin{defi}
An asymptotic basis $A$ of order $h$ is called \emph{thin} if $A(x)\ll x^{1/h}$.
\end{defi}

Thin asymptotic bases exist, and the first explicit examples were constructed independently by Chartrovsky, Raikov,  and St\" ohr in the 1930s.  Thin bases of order $h$ have the property that their counting functions have order of magnitude $x^{1/h}.$   Cassels constructed a family of bases of order $h$ whose counting functions are asymptotic to $\lambda x^{1/h}$ for some positive real number $\lambda.$

\begin{thm}[Cassels~\cite{cass57}]
For every $h\geq 2,$ there exist strictly increasing sequences $A=\{a_n\}_{n=1}^\infty$ of nonnegative integers such that $hA=\N_0$ and $a_n=\lambda n^h+O\left(n^{h-1} \right)$ for some $\lambda > 0.$
\end{thm}

\begin{pbm}[Cassels~\cite{cass57}]
Let $h\geq 2.$  Does there exist an asymptotic basis $A=\{a_n\}_{n=1}^\infty$ of order $h$ such that $a_n=\lambda n^h+o\left(n^{h-1}\right)$ for some $\lambda > 0$?
\end{pbm}

\begin{defi}
A positive real number $\lambda$ will be called an \emph{additive eigenvalue of order $h$} if there exists an asymptotic basis $A =\{a_n\}_{n=1}^\infty$ of order $h$ such that 
\[
a_n \sim \lambda n^h.
\]
We define the \emph{additive spectrum $\Lambda_h$} as the set of all additive eigenvalues of order $h$.
\end{defi}

\begin{thm}[Nathanson~\cite{nath08f}]
For every integer $h \geq 2,$ there is a number $\lambda_h^*$ such that $\Lambda_h = (0, \lambda_h^*)$ or $\Lambda_h = (0, \lambda_h^*].$ 
\end{thm}

The idea of the proof is to show that if $A$ is an asymptotic basis of order $h$ with eigenvalue $\lambda,$ and if $0 < \lambda' < \lambda,$ then one can adjoin nonnegative integers to the set $A$ to obtain an asymptotic basis of order $h$ with eigenvalue $\lambda'.$   Thus, $\Lambda_h$ is an interval.  Combinatorial and geometric arguments show that the additive spectrum is bounded above.

\begin{pbm}
Compute the upper bound $\lambda_h^*$ of the additive spectrum $\Lambda_h.$  Is this upper bound an eigenvalue?
\end{pbm}

\section{Minimal asymptotic bases}

The set $A$ of nonnegative integers is an asymptotic basis of order $h$ if every sufficiently large integer is the sum of exactly $h$ elements of $A$.

\begin{defi}
An asymptotic basis $A$ of order $h$  is \emph{minimal} if, for every element $a^*\in A,$ the set $A\setminus\{a^*\}$ is not an asymptotic basis of order $h$.
\end{defi}

Thus, if $A$ is a minimal asymptotic basis of order $h,$ then for every integer $a^* \in A$ there are infinitely many positive integers $n$ that cannot be represented as the sum of $h$ elements of the set $A\setminus\{a^*\}$.  Equivalently, every element of $A$  is somehow ``responsible'' for the representation of infinitely many integers.

\begin{thm}[H\" artter~\cite{hart56}, Nathanson~\cite{nath74b}]
For every $h \geq 2$ there exist minimal asymptotic bases of order $h$.  
\end{thm}

On the other hand, it is not true that every asymptotic basis of order $h$ contains a subset that is a minimal asymptotic basis of order $h$.
In particular, we have the following result.

\begin{thm}[Erd\H os-Nathanson~\cite{nath75e}]
There exists an asymptotic basis A of order 2 such that, for every
subset $S\subseteq A$, the set $A\setminus S$  is an asymptotic basis of order 2 if and only if  $S$ is finite.
\end{thm}

Since there is no maximal finite subset of an infinite set, it follows that there exists an asymptotic basis of order 2 that contains no minimal asymptotic basis of order 2.

\begin{pbm}
Let $h \geq 3.$  Construct an asymptotic basis $A$ of order $h$ such that, for every subset $S\subseteq A$, the set $A\setminus S$  is an asymptotic basis of order $h$ if and only if $S$ is finite.
\end{pbm}

\begin{pbm}
Find necessary and sufficient conditions to determine if an asymptotic basis A of order $h$ contains a
minimal asymptotic basis of order $h$.
\end{pbm}

In a minimal asymptotic basis every element in the basis is
responsible for the representation of infinitely many numbers. In particular, if $A$ is a minimal asymptotic basis of order 2, then there must be infinitely many positive integers with a unique 
representation as the sum of two elements of $A$.

Let $A$ be a set of integers.  
Let $r_{A,2}(n)$ denote the \emph{unordered representation function} of the set $A$, that is, 
\[
r_{A,2}(n)= \left| \{ \{a_i,a_j\} \subseteq A : n=a_i+a_j \} \right|.
\]

\begin{thm}[Erd\H os-Nathanson~\cite{nath79c}]
Let $A$ be a set of nonnegative integers.  If $r_{A,2}(n)>c \log n$ for some $c>1/ \log(4/3)$ and all $n\geq n_0$, then $A$
contains a minimal asymptotic basis of order 2.
\end{thm}

\begin{pbm}
Is this true if $r_{A,2}(n)>c\log n$ for some $c>0$ ?
\end{pbm}

\begin{pbm}
Let $A$ be a set of nonnegative integers. 
If $r_{A,2}(n)\rightarrow \infty$ as $n\rightarrow \infty$, does
$A$ contain a minimal asymptotic basis of order 2?
\end{pbm}

The idea of minimal asymptotic basis can be generalized in the
following way.

\begin{defi}
Let $r\geq 1$.  The set $A$ is an \emph{$r$-minimal asymptotic basis of order $h$} if, for every $S\subseteq A$, the set 
$A\setminus S$  is an asymptotic basis of order $h$ if and only if $|S|<r.$
\end{defi}

\begin{thm}[Erd\H os-Nathanson~\cite{nath75e}]
For every $r\geq 1$, there exist $r$-minimal asymptotic bases of order 2.
\end{thm}

\begin{pbm}
Let $h \geq 3$ and $r \geq 2.$  Construct an $r$-minimal asymptotic basis $A$ of order $h$. 
\end{pbm}

\section{Maximal asymptotic nonbases}
Maximal asymptotic nonbases are the natural dual to minimal asymptotic bases.

\begin{defi}
The set $A$ of nonnegative integers is an \emph{asymptotic nonbasis  of order $h$} if it is not an asymptotic basis of order $h,$ that is, if $hA$ omits infinitely many nonnegative integers, that is,  the set $\N\setminus hA$ is infinite.
\end{defi}

\begin{defi}
The set $A$ of nonnegative integers is a \emph{maximal asymptotic nonbasis of order $h$} if $A$ is an asymptotic nonbasis of order $h$ such that, for every integer $a^*\in \N\setminus A,$ the set $A\cup\{a^*\}$ is an asymptotic basis of order $h$.
\end{defi}

The construction of minimal asymptotic bases is difficult, but it is easy to find simple examples of maximal nonbases.    For example, the set of all nonnegative even integers is a maximal asymptotic nonbasis of order $h$ for all $h \geq 2.$    The set of all nonnegative multiples of a fixed prime number is a maximal asymptotic nonbasis of order $p.$  Other examples can be constructed by taking appropriate unions of congruence classes.  

\begin{thm}[Nathanson~\cite{nath77b}]
There exist maximal asymptotic nonbases of zero asymptotic density.  
\end{thm}

The follow result implies that there exist asymptotic nonbases that cannot be embedded in maximal asymptotic nonbases.

\begin{thm}[Hennefeld~\cite{henn77}]
There exists an asymptotic nonbasis of order 2 such that, for every set 
$S\subseteq\N\setminus A$, the set $A\cup S$  is an asymptotic  nonbasis of order 2 if and only if 
$ |\N\setminus(A\cup S)|$ is infinite.
\end{thm}

There are many beautiful results on minimal asymptotic bases and
maximal asymptotic nonbases.  Here are two of my favorites.

\begin{thm}[Erd\H os-Nathanson~\cite{nath76c}]
There exists a partition of $\N$ into two sets A and B such that $A$ is a minimal asymptotic basis of order 2 and $B$ is a maximal asymptotic nonbasis of order 2. 
\end{thm}

\begin{thm}[Erd\H os-Nathanson~\cite{nath76c}]
There exists a partition of $\N$ into two sets A and B such that, for any finite subset $F$ of $A$ and any finite subset $G$ of $B$, the partition of $\N$ into the sets 
\[
(A\setminus F) \cup G
\]
and 
\[
(B\setminus G)\cup F
\]
has the following property:  
\begin{enumerate} 
\item[(i)]
If $|F| = |G|$, then $(A\setminus F) \cup G$ is a minimal asymptotic basis of order 2 and $(B\setminus G)\cup F$ is a maximal asymptotic nonbasis of order 2. 
\item[(ii)]
If $|F| = |G|+1$, then $(A\setminus F) \cup G$ is a maximal asymptotic nonbasis of order 2 and $(B\setminus G)\cup F$ is a minimal asymptotic basis of order 2.
\end{enumerate} 
\end{thm}

\section{Complementing sets of integers}

Let $A$ and $B$ be sets of integers, or, more generally, subsets of any additive abelian semigroup $\X$, and let $A+B = \{ a+b : a\in A, b\in B\} = C.$  If every element of the sumset $C$ has a \emph{unique} representation as the sum of an element of $A$ and an element of $B$, then we write $A\oplus B = C.$  We say that the set $A$ \emph{tesselates} the semigroup $\X$ if there exists a set $B$ such that $A\oplus B = \X,$ and that $A$ and $B$ are \emph{complementing subsets of \X.}

In this section we consider complementing sets of integers.  We examine  the special case when $A$ is a finite set of integers, and we want to determine if there exists an infinite set $B$ of integers such that $A\oplus B = \Z.$    By translation, we can always assume that $A$ is a finite set of integers with $0 \in A,$ and that $0$ also belongs to $B$.

We call a set $B$ \emph{periodic} with  \emph{period $m$} if $b\in B$ implies that $b \pm m \in B.$ 

\begin{thm}[D. J. Newmann~\cite{newm77}]
Let $A$  be a finite set of integers.  If there exists a set $B$ such that $A\oplus B=\Z$, then $B$ is periodic with 
period  
\[
m\leq 2^{diam(A)}
\]
where
$diam(A)=\max(A)-\min(A)$.
\end{thm}

It follows that if $A\oplus B = \Z,$ then  $B$ is a union of congruence classes modulo $m$.  Defining 
$\overline{A}=\{a+m\Z :  a\in A\}$ and $\overline{B}=\{b+m\Z : b\in B\}$, we obtain a complementing pair $\overline A\oplus\overline B=\Z/m\Z$.  Conversely, suppose that $\overline{A}$ and $\overline{B}$ are sets of congruence classes modulo $m$ such that $\overline A\oplus\overline B=\Z/m\Z$.  Let $A$ be a set of representatives of the congruence classes in $\overline A$ and let $B$ be the union of the congruence classes in $\overline B$.  Then $ A\oplus B=\Z$.

\begin{thm}[Kolountzakis~\cite{kolo03}, Ruzsa~\cite{ruzs06,tijd06}]
Let $A$ and $B$ be sets of integers such that $A$ is finite, $A\oplus B=\Z$, and $B$ has minimal period $m$.  Then 
\[
m \ll e^{c\sqrt{diam(A)}}.
\]
\end{thm}

\begin{thm}[Biro~\cite{biro05}]
Let $A$ and $B$ be sets of integers such that $A$ is finite, $A\oplus B=\Z$, and $B$ has minimal period $m$.  Then 
\[
m \ll e^{c\sqrt[3]{diam(A)}}.
\]
\end{thm}

\begin{pbm}
Find the least upper bound for the period of a set $B$ of integers that is complementary to a finite set $A$ of diameter $d.$.
\end{pbm}

We can generalize the problem of complementing sets of integers to higher dimensions.
Let $d\geq 2$ and let $A$ be a finite set of lattice points in $\Z^d.$
Suppose there exists a set $B\subseteq \Z^d$ such that $A\oplus B=\Z^d$.  The following problem is well-known.  

\begin{pbm}
Is the set $B$ periodic even in one direction?  Equivalently, does there exist a lattice point $b_0\in\Z^d\setminus\{0\}$ such that $B+ \{ b_0 \} = B$?
\end{pbm}

We can also generalize the problem of complementing sets of integers to linear forms.  Rewrite the original question as follows:
Let $\varphi(x,y)=x+y$.  For sets $A, B\subseteq\Z$, we define the set 
\[
\varphi(A,B)=\{ \varphi(a,b): a\in A,b\in B\}
\]
and the representation function 
\[
r_{A,B}^{(\varphi)}(n)=\{(a,b)\in A\times B : \varphi(a,b)=n\}.
\]
Given a finite set $A$, does there exist a set $B$ such that
$\varphi(A,B)=\Z$ and $r_{A,B}^{(\varphi)}(n)=1$ for all integers $n$?
Now consider the linear forms 
\[
\psi(x_1,\ldots,x_h)=u_1x_1+\ldots+u_hx_h
\]
and 
\[
\varphi(x_1,\ldots,x_h,y)=\psi(x_1,\ldots,x_h)+vy=u_1x_1+\ldots+u_hx_h+vy
\]
with nonzero integer coefficients $u_1,\ldots, u_h,v$.
Given an $h$-tuple $\mathcal{A}=(A_1,\ldots A_h)$ of finite sets of
integers, and a set $B$ of integers, we define the set
\[
\varphi(\mathcal{A},B)=\{u_1a_1+\ldots+u_ha_h+vb :
a_i\in A_i\ \text{ for } i=1,\ldots, h \text{ and }  b\in B\}
\]
and the representation function
\[
r_{\mathcal{A},B}^{(\varphi)}(n)=\{(a_1,\ldots,a_h,b)\in A_1\times\ldots\times A_h\times B : \varphi(a_1,\ldots,a_h,b)=n\}.
\]

\begin{pbm}
Given an $h$-tuple $\mathcal{A}=(A_1,\ldots A_h)$ of finite sets of
integers, determine if there exists a set $B$ such that $\varphi(\mathcal{A},B)=\Z$
and $r_{\mathcal{A},B}^{(\varphi)}(n)=1$ for all integers $n$?
\end{pbm}

In this case, we say that $\mathcal{A}$ and $B$ are complementing sets of integers with respect to the linear form $\varphi.$  

\begin{thm}[Nathanson~\cite{nath08g}]
If $\mathcal{A}$ and $B$ are complementing sets of integers with respect to the linear form $\varphi,$ then B is periodic with period 
\[
m\leq2^{\frac{diam(\psi(A_1,\ldots,A_h))}{|v|}}.
\]
\end{thm}

Ljujic and Nathanson~\cite{nath08h} have extended Biro's cube root upper bound for the period of $m$ to complementing sets of integers with respect to a linear form.

Instead of considering only sets that produce a unique representation for every integer,
we can ask for any prescribed number of representations.  This suggests the following inverse problem for representation functions associated to linear forms:

\begin{pbm}
Let $\varphi(x_1,\ldots,x_h,y)=u_1x_1+\ldots+u_hx_h+vy$ be a linear form with integer coefficients, and let $\mathcal{A}=(A_1,\ldots A_h)$ be an $h$-tuple of finite sets of integers.  Given a function $f:\Z\rightarrow\N$, does there exist a set $B\subseteq\Z$ such
that $r_{\mathcal{A},B}^{(\varphi)}(n) =f(n)$ for all integers $n$?
\end{pbm}

We have the following compactness theorem.

\begin{thm}[Nathanson~\cite{nath08g}]
Let $\varphi(x_1,\ldots,x_h,y)=u_1x_1+\ldots+u_hx_h+vy$ be a linear form with integer coefficients, and let $\mathcal{A}=(A_1,\ldots A_h)$ be an $h$-tuple of finite sets of integers.   Consider a function $f:\Z\rightarrow\N$.
Suppose there exists a strictly increasing sequence $\{k_i\}_{i=1}^{\infty}$ of positive integers and a sequence 
$\{B_i\}_{i=1}^{\infty}$ of (not necessarily increasing) sets of integers such
that $r_{\mathcal{A},B_i}^{(\varphi)}(n)=f(n)$ for integers $n$ satisfying $|n|\leq k_i$. Then there exists an infinite set $B$ such that $r_{\mathcal{A},B}^{(\varphi)}(n)=f(n)$ for all integers $n$.
\end{thm}

\section{The Caccetta-H\" aggkvist conjecture}

Let $G=G(V,E)$ be a finite directed graph with vertex set $V$ and edge set $E$.  Let $n = |V|.$  Every edge $e\in E$ is an ordered pair $(v,v')$ of vertices.  The vertex $v$ is called the \emph{tail} of $e$ and the vertex $v'$ is called the \emph{head} of $e$.  An edge of the form $(v,v)$ is called a \emph{loop}.  We consider only graphs that may have loops, but that do not have multiple edges.  A \emph{path of length $r$} in the graph $G$ is a finite sequence of edges 
$e_1,e_2,\ldots,e_r,$ where $e_i = (v_i,v'_i)$ for $i=1,\ldots, r,$ and $v'_{i}=v_{i+1}$ for $i=1,\ldots, r-1.$   The path is called a \emph{circuit}  if $v'_r=v_1.$  A circuit of length 1 is a loop, a circuit of length 2 is called a \emph{digon}, and a circuit of length 3 is called a \emph{triangle.}

The \emph{outdegree} of a vertex $v$, denoted $\mbox{outdegree}(v),$ is the number of edges $e\in E$ whose tail is $v$.  If $|V| = n$ and $\mbox{outdegree}(v)\geq 1$ for every vertex $v\in V$, then the graph $G$ contains a circuit of length at most $n$.   If $|V| = n$ and $\mbox{outdegree}(v)\geq 2$ for every vertex $v\in V$, then it is known that the graph $G$ contains a circuit of length at most $n/2$.

\begin{conj}[Caccetta-H\" aggkvist]
Let $k \geq 3.$  If $\mbox{outdegree}(v)\geq k$ for every vertex $v\in V$, then the graph $G$ contains a circuit of length at most $n/k$.
Equivalently, if $outdeg(v)\geq n/k$ for every vertex $v\in V$, then the graph $G$ contains a circuit
of length at most $k.$
\end{conj}

Even the case $k=3$ of the Caccetta-H\" aggkvist is open:  If $G$ is a graph with $n$ vertices and if every vertex is the tail of at least $n/3$ edges, prove that the graph contains a loop, a digon, or a directed triangle.  This is a fundamental unsolved problem in graph theory.

\begin{defi}
Let $\Gamma$ be a finite group and let $X\subseteq\Gamma$.  The
Cayley graph $G(V,E)$ is the graph with vertex set $V= \Gamma$ and edge set  $E:=\{(\gamma, \gamma x) : \gamma\in\Gamma, x\in \X\}$.
\end{defi}

\begin{thm}[Hamidoune~\cite{hami81a,hami87a}]
The Caccetta-H\" aggkvist conjecture is true for all Cayley graphs and for all vertex-transitive graphs..
\end{thm}

One proof of this result uses a theorem of Kemperman~\cite{kemp56}  in additive number theory.  An exposition of this and other related results appears in Nathanson~\cite{nath08e}.

\def\cprime{$'$} \def\cprime{$'$} \def\cprime{$'$}
\providecommand{\bysame}{\leavevmode\hbox to3em{\hrulefill}\thinspace}
\providecommand{\MR}{\relax\ifhmode\unskip\space\fi MR }
\providecommand{\MRhref}[2]{%
  \href{http://www.ams.org/mathscinet-getitem?mr=#1}{#2}
}
\providecommand{\href}[2]{#2}


\begin{thebibliography}{10}

\bibitem{biro05}
A.~Bir{\'o}, \emph{Divisibility of integer polynomials and tilings of the
  integers}, Acta Arith. \textbf{118} (2005), no.~2, 117--127.

\bibitem{bukh07}
B.~Bukh, \emph{Sums of dilates}, arXiv preprint 0711.1610, 2007.

\bibitem{cass57}
J.~W.~S. Cassels, \emph{{\" U}ber {B}asen der nat{\" u}rlichen {Z}ahlenreihe},
  Abhandlungen Math. Seminar Univ. Hamburg \textbf{21} (1975), 247--257.

\bibitem{dira51}
G.~A. Dirac, \emph{Note on a problem in additive number theory}, J. London
  Math. Soc. \textbf{26} (1951), 312--313.

\bibitem{erdo-tura41}
P.~Erd\H{o}s and P.~Tur\'an, \emph{On a problem of {S}idon in additive number
  theory, and on some related problems}, J. London Math. Soc. \textbf{16}
  (1941), 212--215.

\bibitem{nath75e}
P.~Erd{\H{o}}s and M.~B. Nathanson, \emph{Oscillations of bases for the natural
  numbers}, Proc. Amer. Math. Soc. \textbf{53} (1975), no.~2, 253--258.

\bibitem{nath76c}
\bysame, \emph{Partitions of the natural numbers into infinitely oscillating
  bases and nonbases}, Comment. Math. Helv. \textbf{51} (1976), no.~2,
  171--182.

\bibitem{nath79c}
\bysame, \emph{Systems of distinct representatives and minimal bases in
  additive number theory}, {Number Theory, Carbondale 1979 (Proc. Southern
  Illinois Conf., Southern Illinois Univ., Carbondale, Ill., 1979)}, Lecture
  Notes in Math., vol. 751, Springer, Berlin, 1979, pp.~89--107.

\bibitem{hami81a}
Y.~O. Hamidoune, \emph{An application of connectivity theory in graphs to
  factorizations of elements in groups}, European J. Combin. \textbf{2} (1981),
  no.~4, 349--355.

\bibitem{hami87a}
\bysame, \emph{A note on minimal directed graphs with given girth}, J. Combin.
  Theory Ser. B \textbf{43} (1987), no.~3, 343--348.

\bibitem{hart56}
E. H{\"a}rtter, \emph{Ein {B}eitrag zur {T}heorie der {M}inimalbasen}, J.
  Reine Angew. Math. \textbf{196} (1956), 170--204.

\bibitem{hega07a}
P.~V. Hegarty, \emph{Some explicit constructions of sets with more sums than
  differences}, Acta Arith. \textbf{130} (2007), 61--77.

\bibitem{hega-mill07}
P.~V. Hegarty and S.~J. Miller, \emph{When almost all sets are difference
  dominated}, arXiv preprint 0707.3417, 2007.

\bibitem{henn77}
J.~Hennefeld, \emph{Asymptotic nonbases which are not subsets of maximal
  aymptotic nonbases}, Proc. Amer. Math. Soc. \textbf{62} (1977), 23--24.

\bibitem{cill-silv-vanu08}
C.~Vinuesa J.~Cilleruelo, M.~Silva, \emph{A sumset problem}, Preprint, 2008.

\bibitem{kemp56}
J.~H.~B. Kemperman, \emph{On complexes in a semigroup}, Indag. Math.
  \textbf{18} (1956), 247--254.

\bibitem{kolo03}
M.~N. Kolountzakis, \emph{Translational tilings of the integers with long
  periods}, Electron. J. Combin. \textbf{10} (2003), Research Paper 22, 9 pp.
  (electronic).

\bibitem{nath08h}
Z.~Ljujic and M.~B. Nathanson, \emph{Complementing sets of integers with
  respect to a multiset}, Preprint, 2008.

\bibitem{mart-obry07a}
G.~Martin and K.~O'Bryant, \emph{Many sets have more sums than differences},
  Additive combinatorics, CRM Proc. Lecture Notes, vol.~43, Amer. Math. Soc.,
  Providence, RI, 2007, pp.~287--305.

\bibitem{nath85a}
J.~C.~M. Nash and M.~B. Nathanson, \emph{Cofinite subsets of asymptotic bases
  for the positive integers}, J. Number Theory \textbf{20} (1985), no.~3,
  363--372.

\bibitem{nath72e}
M.~B. Nathanson, \emph{Sums of finite sets of integers}, Amer. Math. Monthly
  \textbf{79} (1972), 1010--1012.

\bibitem{nath74b}
\bysame, \emph{Minimal bases and maximal nonbases in additive number theory},
  J. Number Theory \textbf{6} (1974), 324--333.

\bibitem{nath77b}
\bysame, \emph{{$s$}-maximal nonbases of density zero}, J. London Math. Soc.
  (2) \textbf{15} (1977), no.~1, 29--34. \MR{MR0435021 (55 \#7983)}

\bibitem{nath03a}
\bysame, \emph{Unique representation bases for the integers}, Acta Arith.
  \textbf{108} (2003), no.~1, 1--8.

\bibitem{nath04a}
\bysame, \emph{The inverse problem for representation functions of additive
  bases}, {Number Theory (New York, 2003)}, Springer, New York, 2004,
  pp.~253--262.

\bibitem{nath05a}
\bysame, \emph{Every function is the representation function of an additive
  basis for the integers}, Port. Math. (N.S.) \textbf{62} (2005), no.~1,
  55--72.

\bibitem{nath08e}
\bysame, \emph{{The Caccetta-H{\"a}ggkvist conjecture and additive number
  theory}}, arXiv: math.CO/0603469, 2006.

\bibitem{nath07b}
\bysame, \emph{Sets with more sums than differences}, Integers \textbf{7}
  (2007), A5, 24 pp. (electronic).

\bibitem{nath08b}
\bysame, \emph{Inverse problems for linear forms over finite sets of integers},
  J. Ramanujan Math. Soc. \textbf{23} (2008), no.~2, 1--15.

\bibitem{nath08g}
\bysame, \emph{{Problems in additive number theory, II: Linear forms and
  complementing sets of integers}}, Journal de Th{\' e}orie des Nombres de
  Bordeaux, to appear, 2008.

\bibitem{nath08f}
\bysame, \emph{Supersequences, rearrangements of sequences, and the spectrum of
  bases in additive number theory}, arXiv preprint 0806.0984, 2008.

\bibitem{nath07d}
M.~B. Nathanson, K.~O'Bryant, B.~Orosz, I.~Ruzsa, and M.~Silva, \emph{Binary
  linear forms over finite sets of integers}, Acta Arith. \textbf{129} (2007),
  341--361.

\bibitem{newm77}
D.~J. Newman, \emph{Tesselation of integers}, J. Number Theory \textbf{9}
  (1977), no.~1, 107--111.

\bibitem{ruzs06}
I.~Z. Ruzsa, \emph{{Appendix in R. Tijdeman, ``Periodicity and
  almost-periodicity''}}, 2006.

\bibitem{tijd06}
R.~Tijdeman, \emph{Periodicity and almost-periodicity}, {More Sets, Graphs and
  Numbers}, Bolyai Soc. Math. Stud., vol.~15, Springer, Berlin, 2006,
  pp.~381--405.

\end{thebibliography}
\end{document}